\documentclass[10pt]{article}

\usepackage{amsthm,amsmath,amssymb}

\usepackage{graphicx}

\usepackage[colorlinks=true,citecolor=blue,linkcolor=blue,urlcolor=blue]{hyperref}

\theoremstyle{plain}
\newtheorem{theorem}{Theorem}[section]

\newtheorem{proposition}[theorem]{Proposition}

\theoremstyle{definition}

\newtheorem{example}[theorem]{Example}

\theoremstyle{remark}

\title{\bf Counting faces of graphical zonotopes}

\author{Vladimir Gruji\'c\\
\small Faculty of Mathematics\\[-0.8ex]
\small Belgrade University, Serbia\\[-0.8ex]}

\date{
\small Mathematics Subject Classifications: 05E05, 52B05, 16T05}

\begin{document}

\maketitle

\begin{abstract}
It is a classical fact that the number of vertices of the
graphical zonotope $Z_\Gamma$ is equal to the number of acyclic
orientations of a graph $\Gamma$. We show that the $f$-polynomial
of $Z_\Gamma$ is obtained as the principal specialization of the
$q$-analog of the chromatic symmetric function of $\Gamma$.

\bigskip\noindent\textbf{Keywords}: graphical zonotope, $f$-vector,
 graphical matroid, symmetric function
\end{abstract}

\section{Introduction}

The $f$-polynomial of an $n$-dimensional polytope $P$ is defined
by $f(P,q)=\sum_{i=0}^{n}f_i(P)q^{i},$ where $f_i(P)$ is the
number of $i$-dimensional faces of $P$. The $f$-polynomial
$f(\mathcal{Z}_\Gamma, q)$ of the graphical zonotope
$\mathcal{Z}_\Gamma$ is a combinatorial invariant of a finite,
simple graph $\Gamma$. The vertices of $\mathcal{Z}_\Gamma$ are in
one-to-one correspondence with regions of the graphical hyperplane
arrangement $\mathcal{H}_\Gamma$, which are enumerated by acyclic
orientations of $\Gamma$.

Stanley's chromatic symmetric function $\Psi(\Gamma)=\sum_{f
{\small\it proper}}\mathbf{x}_f$ of a graph $\Gamma=(V,E)$,
introduced in \cite{S}, is the enumerator function of proper
colorings $f:V\rightarrow\mathbb{N}$, where
$\mathbf{x}_f=x_{f(1)}\cdots x_{f(n)}$ and $f$ is proper if there
are no monochromatic edges.  The chromatic polynomial
$\chi(\Gamma, d)$ of the graph $\Gamma$, which counts proper
colorings with a finite number of colors, appears as the principal
specialization

$$\chi(\Gamma,d)=\mathbf{ps}(\Psi(\Gamma))(d)=\Psi(\Gamma)\mid_{x_1=\cdots=x_d=1,
x_{d+1}=\cdots=0}.$$ The number of acyclic orientations of
$\Gamma$ is determined by the value of the chromatic polynomial
$\chi(\Gamma,d)$ at $d=-1$, \cite{S1}

\begin{equation}\label{equalities}
a(\Gamma)=(-1)^{|V|}\chi(\Gamma,-1).
\end{equation}

There is a $q$-analog of the chromatic symmetric function
$\Psi_q(\Gamma)$ introduced in a wider context of the
combinatorial Hopf algebra of simplicial complexes considered in
\cite{BHM}. It is a symmetric function over the field of rational
functions in $q$. The principal specialization of $\Psi_q(\Gamma)$
is the $q$-analog of the chromatic polynomial $\chi_q(\Gamma,d)$.

The main result of this paper is the following generalization of
formula $(\ref{equalities})$

\begin{theorem}\label{main}
Let $\Gamma=(V,E)$ be a simple connected graph and
$\mathcal{Z}_\Gamma$ the corresponding graphical zonotope. Then
the $f$-polynomial of $\mathcal{Z}_\Gamma$ is given by
$$f(\mathcal{Z}_\Gamma,q)=(-1)^{|V|}\chi_{-q}(\Gamma,-1).$$
\end{theorem}

The cancellation-free formula for the antipode in the Hopf algebra
of graphs, obtained by Humpert and Martin in \cite{HM}, reflects
the fact that $f(\mathcal{Z}_\Gamma,q)$ depends only on the
graphical matroid $M(\Gamma)$ associated to $\Gamma$. For
instance, for any tree $T_n$ the graphical matroid is the uniform
matroid $M(T_n)=U_n^{n}$ and the corresponding graphical zonotope
is the cube $\mathcal{Z}_{T_n}=I^{n-1}$. Whitney's theorem from
1933 describes how two graphs with the same graphical matroid are
related \cite{W}. It can be used to find more interesting
nonisomorphic graphs with the same $f$-polynomials of
corresponding graphical zonotopes.

The paper is organized as follows. In section 2, we review the
basic facts about zonotopes. In section 3, the $q$-analog of the
chromatic symmetric function $\Psi_q(\Gamma)$ of a graph $\Gamma$
is introduced. Theorem \ref{main} is proved in section 4. We
present some examples and calculations in section 5.

\section{Zonotopes}

A {\it zonotope} $\mathcal{Z}=\mathcal{Z}(v_1,\ldots,v_m)$ is a
convex polytope determined by a collection of vectors
$\{v_1,\ldots,v_m\}$ in $\mathbb{R}^{n}$ as the Minkowski sum of
line segments
$$\mathcal{Z}=[-v_1,v_1]+\cdots+[-v_m,v_m].$$ It is a
projection of the $m$-cube $[-1,1]^{m}$ under the linear map
$\mathbf{t}\mapsto A\mathbf{t}, \mathbf{t}\in[-1,1]^{m}$, where
$A=[v_1\cdots v_m]$ is an $n\times m$-matrix whose columns are
vectors $v_1,\ldots,v_m$. The zonotope $\mathcal{Z}$ is symmetric
about the origin and all its faces are translations of zonotopes.

To a collection of vectors $\{v_1,\ldots,v_m\}$ is associated a
central arrangement of hyperplanes
$\mathcal{H}=\{H_{v_1},\ldots,H_{v_m}\}$, where $H_v$ denotes the
hyperplane perpendicular to a vector $v\in\mathbb{R}^{n}$. The
zonotope $\mathcal{Z}$ and the corresponding arrangement of
hyperplanes $\mathcal{H}$ are closely related. In fact the
associated fan $\mathcal{F}_\mathcal{H}$ of the arrangement
$\mathcal{H}$ is the normal fan $\mathcal{N}(\mathcal{Z})$ of the
zonotope $\mathcal{Z}$ (see \cite[Theorem 7.16]{Z}). It follows
that the face lattice of $\mathcal{F}_\mathcal{H}$ and the reverse
face lattice of $\mathcal{Z}$ are isomorphic. In particular,
vertices of $\mathcal{Z}$ correspond to regions of $\mathcal{H}$
and their total numbers coincide

\begin{equation}\label{vertices-regions}
f_0(\mathcal{Z})=r(\mathcal{H}).
\end{equation}

The faces of the zonotope $\mathcal{Z}$ are encoded by covectors
of the oriented matroid $\mathcal{M}$ associated to the collection
of vectors $\{v_1,\ldots,v_m\}$. The covectors are sign vectors

$$\mathcal{V}^{\ast}=\{\mathrm{sign}(v)\in\{+,-,0\}^{m}\mid v\in
\mathbb{R}^{n}\},$$ where
$\mathrm{sign}(v)_i=\left\{\begin{array}{cc}+, & \langle
v,v_i\rangle>0 \\ 0, & \langle v,v_i\rangle=0
\\ -, & \langle v,v_i\rangle<0\end{array}\right., \ i=1,\ldots,m.$
The face lattice of the zonotope $\mathcal{Z}$ is isomorphic to
the lattice of covectors componentwise induced by $+,-<0$ on
$\mathcal{V}^{\ast}$.

A special class of zonotopes is determined by simple graphs. To a
connected graph $\Gamma=(V,E)$, whose vertices are enumerated by
integers $V=\{1,\ldots,n\}$, are associated the {\it graphical
zonotope}
$$\mathcal{Z}_\Gamma=\mathcal{Z}(e_i-e_j\mid i<j, \{i,j\}\in E)$$
and the {\it graphical arrangement} in $\mathbb{R}^{n}$
$$\mathcal{H}_\Gamma=\{H_{e_i-e_j}\mid i<j, \{i,j\}\in E\}.$$
There is a bijective correspondence between regions of
$\mathcal{H}_\Gamma$ and acyclic orientations of $\Gamma$,
\cite[Proposition 2.5]{S2}, which by $(\ref{vertices-regions})$
implies

\begin{equation}\label{vertices-acyclic}
f_0(\mathcal{Z}_\Gamma)=r(\mathcal{H}_\Gamma)=a(\Gamma).
\end{equation}
The arrangement $\mathcal{H}_\Gamma$ is refined by the braid
arrangement $\mathcal{A}_{n-1}$ consisting of all hyperplanes
$H_{e_i-e_j}, 1\leq i<j\leq n$. Thus $\mathcal{Z}_\Gamma$ belongs
to a wider class of convex polytopes called generalized
permutohedra introduced in \cite{P}. Since arrangements
$\mathcal{H}_\Gamma$ and $\mathcal{A}_{n-1}$ are not essential we
take their quotients by the line $l: x_1=\cdots=x_n$ and without
confusing retain the same notation. Consequently
$\mathrm{dim}\mathcal{Z}_\Gamma=n-1$.

\begin{figure}[h!h!]
\centerline{\includegraphics[width=8cm]{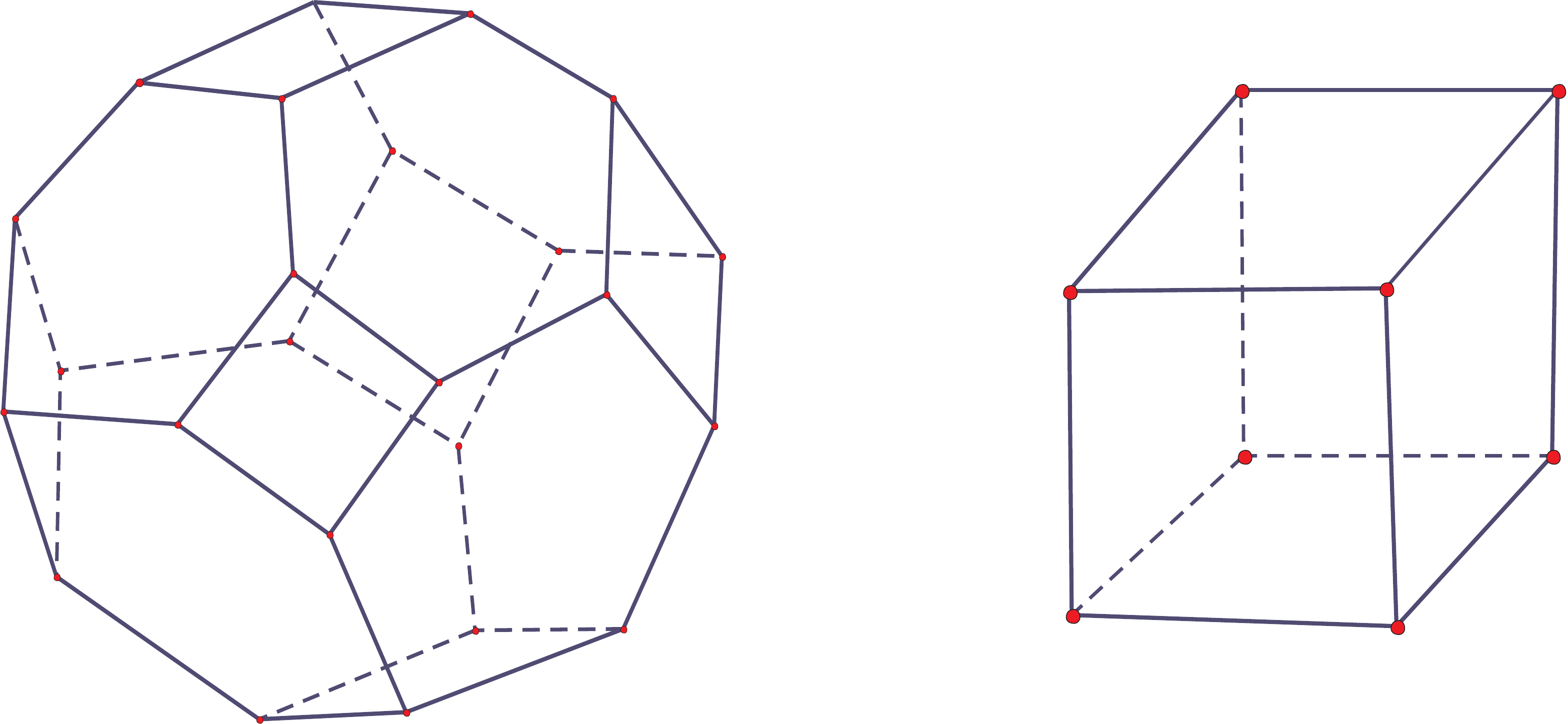}}
\caption{Permutohedron $Pe^{3}$ and cube $I^{3}$}\label{pc}
\end{figure}

\begin{example}\label{perm-cub}
(i) The permutohedron $Pe^{n-1}$ is represented as the graphical
zonotope $\mathcal{Z}_{K_n}$ corresponding to the complete graph
$K_n$ on $n$ vertices (Figure \ref{pc}).

(ii) The cube $I^{n-1}$ is represented as the graphical zonotope
$\mathcal{Z}_{T_n}$ corresponding to an arbitrary tree $T_n$ on
$n$ vertices. This shows that the graph $\Gamma$ is not determined
by the combinatorial type of the zonotope $\mathcal{Z}_\Gamma$.
\end{example}

\section{$q$-analog of chromatic symmetric function of graph}

Stanley's chromatic symmetric function $\Psi(\Gamma)$ can be
obtained in a purely algebraic way. A combinatorial Hopf algebra
$\mathcal{H}$ is a graded, connected Hopf algebra equipped with
the multiplicative linear functional
$\zeta:\mathcal{H}\rightarrow\mathbf{k}$ to the ground field
$\mathbf{k}$. For the theory of combinatorial Hopf algebras see
\cite{ABS}. Consider the combinatorial Hopf algebra of graphs
$\mathcal{G}$ which is linearly generated over a field
$\mathbf{k}$ by simple finite graphs with the product defined by
disjoint union $\Gamma_1\cdot\Gamma_2=\Gamma_1\sqcup\Gamma_2$ and
the coproduct
$$\Delta(\Gamma)=\sum_{I\subset
V}\Gamma\mid_I\otimes\Gamma\mid_{V\setminus I},$$ where
$\Gamma\mid_I$ denotes the induced subgraph on $I\subset V$. The
structure of $\mathcal{G}$ is completed by the character
$\zeta:\mathcal{G}\rightarrow\mathbf{k}$ defined to be
$\zeta(\Gamma)=1$ for $\Gamma$ with no edges and $\zeta(\Gamma)=0$
otherwise. Then it turns out that $\Psi(\Gamma)$ is the image of
the unique morphism of combinatorial Hopf algebras to symmetric
functions $\Psi:\mathcal{G}\rightarrow Sym$, (\cite[Example
4.5.]{ABS}).

An important part of the structure of the Hopf algebra
$\mathcal{G}$ is the antipode
$S:\mathcal{G}\rightarrow\mathcal{G}$. The cancellation-free
formula for the antipode in terms of acyclic orientations of a
graph $\Gamma$ is obtained in \cite{HM}. We recall some basic
definitions. Terminology comes from matroid theory. Given a graph
$\Gamma=(V,E)$, for a collection of edges $F\subset E$ denote by
$\Gamma_{V,F}$ the graph on $V$ with the edge set $F$. A {\it
flat} $F$ of the graph $\Gamma$ is a collection of its edges such
that components of $\Gamma_{V,F}$ are induced subgraphs. The rank
$\mathrm{rk}(F)$ is the size of spanning forests of
$\Gamma_{V,F}$. We have that $|V|=\mathrm{rk}(F)+c(F)$, where
$c(F)$ is the number of components of $\Gamma_{V,F}$. By
contracting edges from a flat $F$ we obtain the graph $\Gamma/F$.
Finally, let $a(\Gamma)$ be the number of acyclic orientations of
$\Gamma$. The formula of Humpert and Martin is as follows

\begin{equation}\label{antipode}
S(\Gamma)=\sum_{F\in\mathcal{F}(\Gamma)}(-1)^{c(F)}a(\Gamma/F)\Gamma_{V,F},
\end{equation}
where the sum is over the set of flats $\mathcal{F}(\Gamma)$.

The following modification of the character $\zeta$ is considered
in \cite{BHM} in a wider context of the combinatorial Hopf algebra
of simplicial complexes. Define
$\zeta_q(\Gamma)=q^{\mathrm{rk}(\Gamma)}$, which determines the
algebra morphism $\zeta_q:\mathcal{G}\rightarrow\mathbf{k}(q)$,
where $\mathbf{k}(q)$ is the field of rational functions in $q$.
This character produces the unique morphism
$\Psi_q:\mathcal{G}\rightarrow QSym$ to quasisymmetric functions
over $\mathbf{k}(q)$. The expansion of $\Psi_q(\Gamma)$ in the
monomial basis of quasisymmetric functions is determined by the
universal formula \cite[Theorem 4.1]{ABS}

$$\Psi_q(\Gamma)=\sum_{\alpha\models
n}(\zeta_q)_\alpha(\Gamma)M_\alpha.$$ The sum above is over all
compositions of the integer $n=|V|$ and the coefficient of the
expansion corresponding to the composition
$\alpha=(a_1,\ldots,a_k)\models n$ is given by

$$(\zeta_q)_\alpha(\Gamma)=\sum_{I_1\sqcup\ldots\sqcup
I_k=V}q^{\mathrm{rk}(\Gamma\mid_{I_1})+\cdots+\mathrm{rk}(\Gamma\mid_{I_k})},$$
where the sum is over all set compositions of $V$ of the type
$\alpha$. The coefficients $(\zeta_q)_\alpha(\Gamma)$ depend only
on the partition corresponding to a composition $\alpha$, so the
function $\Psi_q(\Gamma)$ is actually symmetric and it can be
expressed in the monomial basis of symmetric functions.

The invariant $\Psi_q(\Gamma)$ is more subtle than $\Psi(\Gamma)$.
Obviously $\Psi_0(\Gamma)$ is the chromatic symmetric function of
a graph $\Gamma$. It remains open to find two nonisomorphic graphs
$\Gamma_1$ and $\Gamma_2$ with the same $q$-chromatic symmetric
functions $\Psi_q(\Gamma_1)=\Psi_q(\Gamma_2)$. Let
$$\chi_q(\Gamma,d)=\mathbf{ps}(\Psi_q(\Gamma))(d)$$ be the $q$-analog
of the chromatic polynomial $\chi(\Gamma,d)$. It is a consequence
of a general fact for combinatorial Hopf algebras (see \cite{ABS})
that

\begin{equation}\label{inverse}
\chi_q(\Gamma,-1)=(\zeta_q\circ S)(\Gamma).
\end{equation}

\begin{example}\label{exm1}

Consider the graph $\Gamma$ on four vertices with the edge set
$E=\{12,13,23,34\}$. We find that
$$\Psi_q(\Gamma)=24m_{1,1,1,1}+(8q+4)m_{2,1,1}+(2q^{2}+4q)m_{2,2}+(3q^{2}+q)m_{3,1}+q^{3}m_4.$$
By principal specialization and taking into account that
$$\mathbf{ps}(m_{\lambda_1^{i_1},\ldots,\lambda_k^{i_k}})(d)=
\frac{(i_1+\cdots+i_k)!}{i_1!\cdots i_k!}{d\choose
i_1+\cdots+i_k},$$ we obtain

$$\chi_q(\Gamma,d)=d(d-1)^{2}(d-2)+qd(d-1)(4d-5)+4q^{2}d(d-1)+q^{3}d,$$
which by Theorem \ref{main} gives

$$f(\mathcal{Z}_\Gamma,q)=12+18q+8q^{2}+q^{3}.$$
\end{example}

\section{Proof of Theorem \ref{main}}

By applying $(\ref{inverse})$ and the formula for antipode
$(\ref{antipode})$ we obtain

$$(-1)^{|V|}\chi_{-q}(\Gamma,-1)=(-1)^{|V|}\sum_{F\in\mathcal{F}(\Gamma)}(-1)^{c(\Gamma)}a(\Gamma/F)(-q)^{\mathrm{rk}(F)}.$$
It follows that the statement of the theorem is equivalent to the
following expression of the $f$-polynomial

\begin{equation}\label{f-polynomial}
f(\mathcal{Z}_\Gamma,q)=\sum_{F\in\mathcal{F}(\Gamma)}a(\Gamma/F)q^{\mathrm{rk}(F)}.
\end{equation}
Therefore it should be shown that components of $f$-vectors are
determined by

\begin{equation}\label{f-vector}
f_k(\mathcal{Z}_\Gamma)=\sum_{\begin{array}{cc}F\in\mathcal{F}(\Gamma)\\
\mathrm{rk}(F)=k \end{array}}a(\Gamma/F), \ \ 0\leq k\leq n-1.
\end{equation}
By duality between the face lattice of $\mathcal{Z}_\Gamma$ and
the face lattice of the fan $\mathcal{F}_{\mathcal{H}_\Gamma}$ we
have

$$f_k(\mathcal{Z}_\Gamma)=f_{n-k-1}(\mathcal{F}_{\mathcal{H}_\Gamma}).$$

Let $L(\mathcal{H}_\Gamma)$ be the intersection lattice of the
graphical arrangement $\mathcal{H}_\Gamma$. For a subspace $X\in
L(\mathcal{H}_\Gamma)$ there is an arrangement of hyperplanes

$$\mathcal{H}_\Gamma^{X}=\{X\cap H\mid X\nsubseteq H,
H\in\mathcal{H}_\Gamma\}$$ whose intersection lattice
$L(\mathcal{H}_\Gamma^{X})$ is isomorphic to the upper cone of $X$
in $L(\mathcal{H}_\Gamma)$. Since $\mathcal{H}_\Gamma$ is central
and essential we have

\begin{equation}\label{regions}
f_{n-k-1}(\mathcal{F}_{\mathcal{H}_\Gamma})=\sum_{\begin{array}{cc}X\in
L(\mathcal{H}_\Gamma)\\ \mathrm{dim}(X)=n-k-1
\end{array}}r(\mathcal{H}^{X}_\Gamma),
\end{equation}
where $r(\mathcal{H}^{X}_\Gamma)$ is the number of regions of the
arrangement $\mathcal{H}^{X}_\Gamma$, see \cite[Theorem 2.6]{S2}.

The intersection lattice $L(\mathcal{H}_\Gamma)$ is isomorphic to
the lattice of flats of the graphical matroid $M(\Gamma)$. By this
isomorphism to a flat $F$ of rank $k$ corresponds the intersection
subspace $X^{F}=\cap_{\{i,j\}\in F}H_{e_i-e_j}$ of dimension
$n-k-1$. It is easy to see that arrangements
$\mathcal{H}_\Gamma^{X^{F}}$ and $\mathcal{H}_{\Gamma/F}$
coincide, which by $(\ref{vertices-acyclic})$ and comparing
formulas $(\ref{f-vector})$ and $(\ref{regions})$ proves theorem.

\section{Examples}

By applying Theorem \ref{main} we obtain the following
interpretation of identities elaborated in \cite[Propositions 17,
19]{BHM}.

\begin{example}\label{perm-cube}
(i) For the permutohedron $Pe^{n-1}=\mathcal{Z}_{K_n}$, the
$f$-polynomial is given by
$$f(\mathcal{Z}_{K_n},q)=A_n(q+1),$$ where $A_n(q)=\sum_{\pi\in
S_n}q^{\mathrm{des}(\pi)}$ is the Euler polynomial. Recall that
$\mathrm{des}(\pi)$ is the number of descents of a permutation
$\pi\in S_n$. It recovers the fact that the $h$-polynomial of the
permutohedron $Pe^{n-1}$ is the Euler polynomial $A_n(q)$.

(ii) For the cube $I^{n-1}=\mathcal{Z}_{T_n}$, where $T_n$ is a
tree on $n$ vertices, the $f$-polynomial is given by

$$f(\mathcal{Z}_{T_n},q)=(q+2)^{n-1}.$$

\end{example}

\begin{figure}[h!h!]
\centerline{\includegraphics[width=10cm]{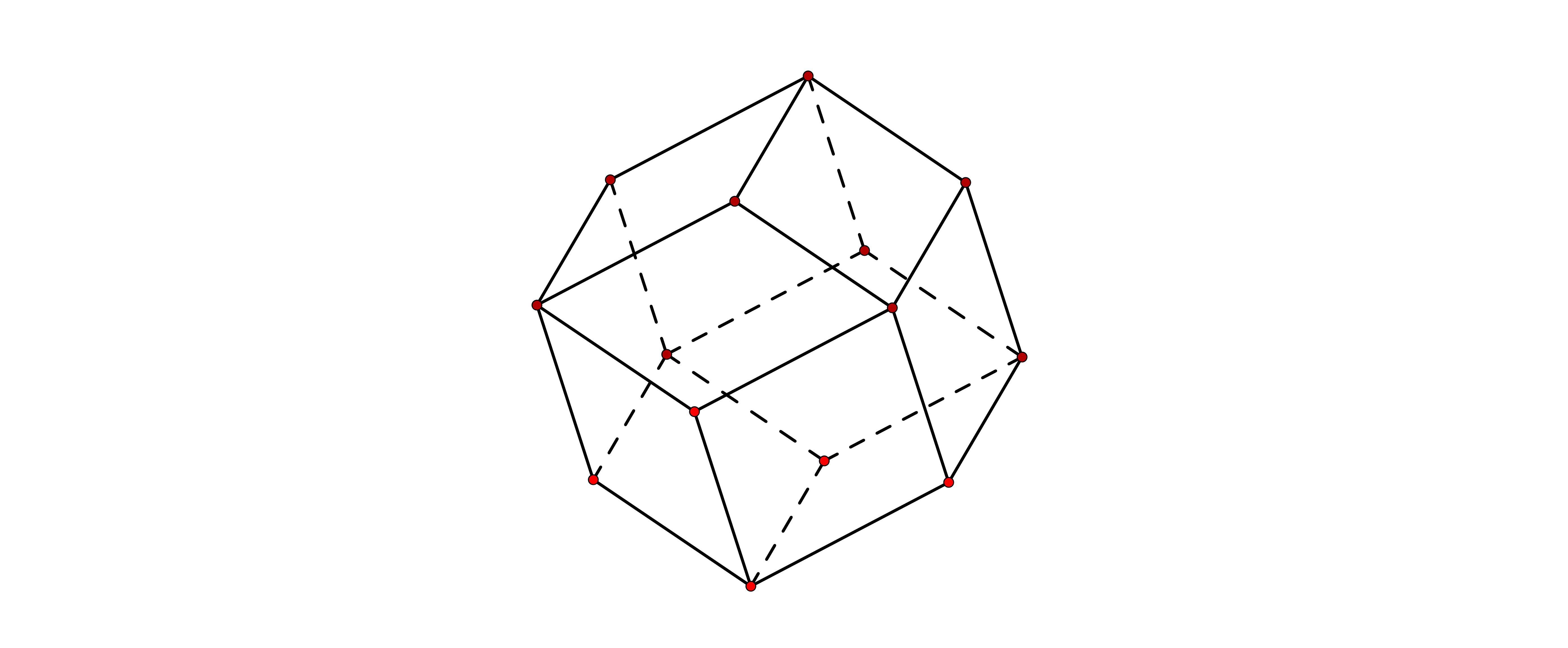}} \caption{Rhombic
dodecahedron $\mathcal{Z}_{C_4}$} \label{rd}
\end{figure}

\begin{proposition}
The $f$-polynomial of the graphical zonotope $\mathcal{Z}_{C_n}$
associated to the cycle graph $C_n$ on $n$ vertices is given by

$$f(\mathcal{Z}_{C_n},q)=q^{n}+q^{n-1}+(q+2)^{n}-2(q+1)^{n}.$$
\end{proposition}
\begin{proof}
A flat $F\in\mathcal{F}(C_n)$ is determined by the complementary
set of edges. If $\mathrm{rk}(F)=n-k, k>1$ then the complementary
set has $k$ edges and $C_n/F=C_k$. Since $a(C_k)=2^{k}-2,k>1$, by
formula $(\ref{f-vector})$, we obtain

$$f_{n-k}(\mathcal{Z}_{C_n})=(2^{k}-2){n\choose k}, 2\leq k\leq n,$$
which leads to the required formula.

\end{proof}
Specially, for $n=4$ the resulting zonotope is the rhombic
dodecahedron (see Figure \ref{rd}). We have

$$f(\mathcal{Z}_{C_4},q)=14+24q+12q^{2}+q^{3}.$$

\begin{proposition}\label{vee}
Let $\Gamma=\Gamma_1\vee_v\Gamma_2$ be the wedge of two connected
graphs $\Gamma_1$ and $\Gamma_2$ at the common vertex $v$. Then
$$f(\mathcal{Z}_\Gamma,q)=f(\mathcal{Z}_{\Gamma_1},q)f(\mathcal{Z}_{\Gamma_2},q).$$
\end{proposition}

\begin{proof}
The graphical matroids of involving graphs are related by
$M(\Gamma)=M(\Gamma_1)\oplus M(\Gamma_2)$. For the sets of flats
it holds $\mathcal{F}(\Gamma)=\{F_1\cup F_2\mid
F_i\in\mathcal{F}(\Gamma_i),i=1,2\}$. For $F=F_1\cup F_2$ we have
$\Gamma/F=\Gamma_1/F_1\vee_{[v]}\Gamma_2/F_2,$ where $[v]$ is the
component of the vertex $v$ in $\Gamma_{V,F}$. Obviously
$a(\Gamma/F)=a(\Gamma_1/F_1)a(\Gamma_2/F_2)$ and
$\mathrm{rk}(F)=\mathrm{rk}(F_1)+\mathrm{rk}(F_2)$. The
proposition follows from formula $(\ref{f-polynomial})$.
\end{proof}

The formula for cubes in Example $\ref{perm-cube}$ (ii) follows
from Proposition $\ref{vee}$ since any tree is a consecutive wedge
of edges and $f(I^{1},q)=q+2$. It also allows us to restrict
ourselves only to biconnected graphs. For a biconnected graph
$\Gamma$ with a disconnecting pair of vertices $\{u,v\}$ Whitney
introduced the transformation called the {\it twist} around the
pair $\{u,v\}$. This transformation does not have an affect on the
graphical matroid $M(\Gamma)$ \cite{W}.

\begin{figure}[h!h!]
\centerline{\includegraphics[width=7cm]{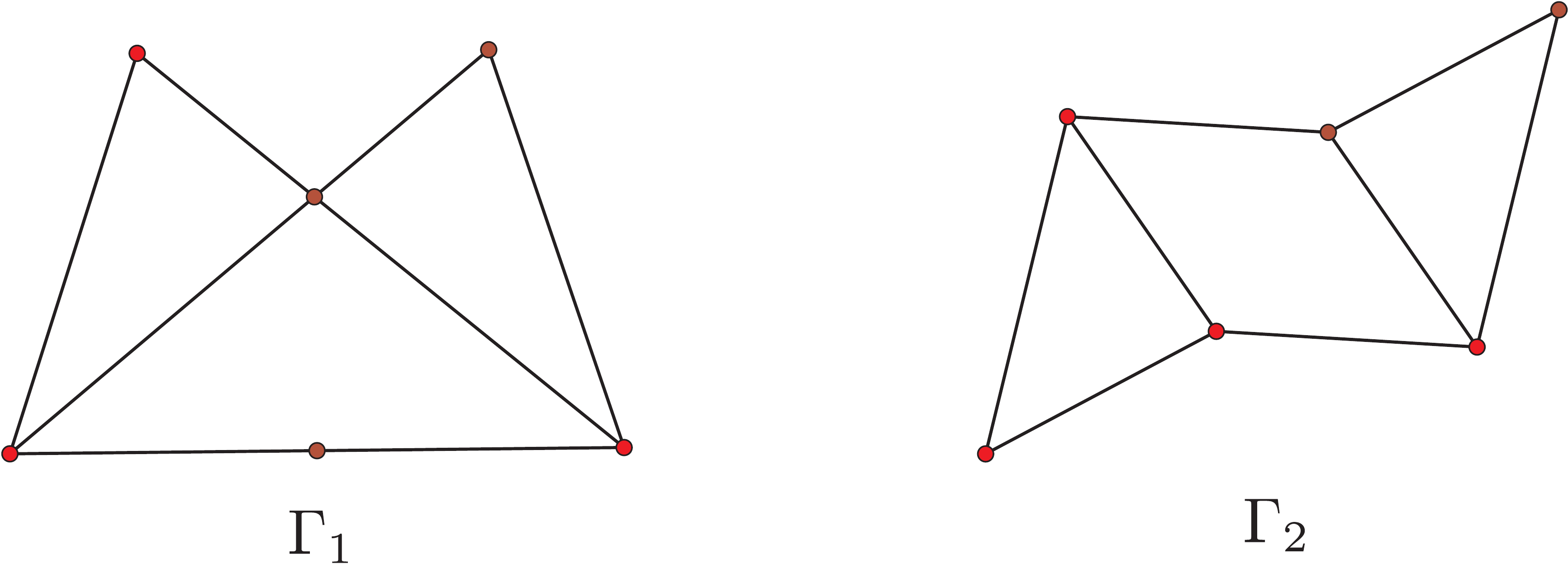}}
\caption{Biconnected graphs related by twist transformation}
\label{twographs}
\end{figure}

\begin{example}
Figure \ref{twographs} shows the pair of biconnected graphs on six
vertices obtained one from another by the twist transformation.
The corresponding zonotopes have the same $f$-polynomial

$$f(\mathcal{Z}_{\Gamma_1},q)=f(\mathcal{Z}_{\Gamma_2},q)=126+348q+358q^{2}+164q^{3}+30q^{4}+q^{5}.$$
On the other hand their $q$-chromatic symmetric functions are
different. One can check that corresponding coefficients by
$m_{3,1^{3}}$ are different

$$[m_{3,1^{3}}]\Psi_q(\Gamma_1)=(11q^{2}+8q+1)\cdot 3!,$$
$$[m_{3,1^{3}}]\Psi_q(\Gamma_2)=(10q^{2}+10q)\cdot 3!.$$
This shows that the $q$-analog of the chromatic symmetric function
of a graph is not determined by the corresponding graphical
matroid. By taking $q=0$ we obtain that even the chromatic
symmetric functions are different since
$[m_{3,1^{3}}]\Psi(\Gamma_1)=6$ and
$[m_{3,1^{3}}]\Psi(\Gamma_2)=0$.

Let us now consider Stanley's example of nonisomorphic graphs with
the same chromatic symmetric functions, see \cite{S}. We find that
the $f$-polynomials of the corresponding graphical zonotopes
differ for those graphs. From these examples we conclude that
chromatic properties of a graph and the $f$-vector of the
corresponding graphical zonotope are not related.

\end{example}

We have already noted that graphical zonotopes are generalized
permutohedra. The $h$-polynomials of simple generalized
permutohedra are determined in \cite[Theorem 4.2]{PRW}. The only
simple graphical zonotopes are products of permutohedra
\cite[Proposition 5.2]{PRW}. They are characterized by graphs
whose biconnected components are complete subgraphs. Therefore
Proposition \ref{vee} together with Example \ref{perm-cube} (i)
prove that the $h$-polynomial of a simple graphical zonotope is
the product of Eulerian polynomials, the fact obtained in
\cite[Corollary 5.4]{PRW}. Example \ref{exm1} is of this sort and
represents the hexagonal prism which is the product
$\mathcal{Z}_{K_3}\times\mathcal{Z}_{K_2}$.

\bibliographystyle{amsplain}

\begin{thebibliography}{7}

\bibitem{ABS} M. Aguiar, N. Bergeron and F. Sottile, Combinatorial Hopf
algebras and generalized Dehn-Sommerville relations,
\emph{Compositio Math.} \textbf{142} (2006), 1--30.

\bibitem{BHM} C. Benedetti, J. Hallam and J. Machacek,
Combinatorial Hopf algebra of simplicial complexes, \emph{SIAM J.
Discrete Math.}, \textbf{30}(3) (2016), 1737--1757.

\bibitem{HM} B. Humpert and J. Martin, The incidence Hopf algebra of
graphs, \emph{SIAM J. Discrete Math.}, \textbf{26} (2) (2012),
555–--570.

\bibitem{P} A. Postnikov, Permutohedra, associahedra, and beyond,
\emph{Int. Math. Res. Not.} 6 (2009), 1026--1106.

\bibitem{PRW} A. Postnikov, V. Reiner and L. Williams, Faces of
generalized permutohedra, \emph{Documenta Math.} \textbf{13}
(2008), 207--273.

\bibitem{S} R. Stanley, A symmetric function generalization of the
chromatic polynomial of a graph, \emph{Adv. Math.} \textbf{111}
(1995), 166--194.

\bibitem{S1} R. Stanley, Acyclic orientations of
graphs, \emph{Discrete Math.} \textbf{5} (1973), 171--178.

\bibitem{S2} R. Stanley, \emph{An introduction to hyperplane arrangements}, in:
E. Miller, V. Reiner, B. Sturmfels (eds.), \emph{Geometric
Combinatorics}, IAS/Park City Mathematics Series, vol. 13, AMS,
Institute for Advance Studies, 2007, 389--496.

\bibitem{W} H. Whitney, 2-isomorphic graphs, \emph{Amer. J. Math.} \textbf{55} (1933), 245--254

\bibitem{Z} G. Ziegler, \emph{Lectures on Polytopes}, Springer-Verlag,
New York, 1995.

\bibliographystyle{amsplain}
\end{thebibliography}

\end{document}